\documentclass[12pt]{amsart}
\usepackage[alphabetic]{amsrefs}
\usepackage{graphicx}
\usepackage{tikz}
\usepackage{amscd}
\usepackage{upgreek}
\usepackage{stmaryrd}
\usepackage{longtable}
\usepackage[T1]{fontenc}
\usepackage{latexsym, amsmath, amssymb, amsthm}
\usepackage{rsfso}
\usepackage[mathscr]{eucal}
\usepackage{mathtools}
\usepackage{mathptmx}
\usepackage{titletoc}
\usepackage{wrapfig}
\usepackage{float}
\usepackage{xypic}
\usepackage{microtype}
\usepackage{dsfont}
\usepackage{xcolor}
\usepackage{color}
\usepackage[colorlinks,linkcolor=blue,anchorcolor=blue,urlcolor=blue,
citecolor=blue]{hyperref}

\usepackage{hyperref}
\allowdisplaybreaks	
\usepackage[centering, includeheadfoot, hmargin=1.0in, tmargin=0.5in, 
  bmargin=0.6in, headheight=6pt]{geometry}
\newtheorem{theorem}{Theorem}[section]
\newtheorem{prop}[theorem]{Proposition}

\newtheorem{lem}[theorem]{Lemma}
\newtheorem{coro}[theorem]{Corollary}

\newtheorem{thm}[theorem]{Theorem}

\newtheorem{exam}[theorem]{Example}

\DeclareMathOperator{\GL}{GL}

\DeclareMathOperator{\diag}{diag}

\newcommand{\N}{\mathbb{N}} 
\newcommand{\F}{\mathbb{F}} 

\newcommand{\ra}{\longrightarrow} 
\newcommand{\A}{\mathcal{A}} 
\newcommand{\B}{\mathcal{B}} 
\newcommand{\OO}{\mathcal{O}} 
 
\newcommand{\NN}{\mathcal{N}} 
\newcommand{\M}{\mathcal{M}} 
\newcommand{\I}{\mathcal{I}} 
\newcommand{\V}{\mathcal{V}} 

\newcommand{\wt}{\widetilde}
 
\newcommand{\hbo}{$\hfill\Diamond$} 

\begin{document}
\title{A class of quadratic matrix equations over finite fields} 
\def\shorttitle{A class of quadratic matrix equations over finite fields}

\author{Yin Chen}
\address{School of Mathematics and Statistics, Northeast Normal University, Changchun, China \& Department of Mathematics and Statistics, Queen's University, Kingston, K7L 3N6, Canada}
\email{ychen@nenu.edu.cn}

\author{Xinxin Zhang}
\address{School of Mathematics and Statistics, Northeast Normal University, Changchun, China}
\email{zhangxx272@nenu.edu.cn}

\begin{abstract}
We exhibit an explicit formula for the cardinality of solutions to a class of quadratic matrix equations over finite fields. 
We prove that the orbits of these solutions under the natural conjugation action of the general linear groups can be separated by classical conjugation invariants defined by characteristic polynomials. We also find a generating set for the vanishing ideal of these orbits.
\end{abstract}

\date{\today}
\maketitle \baselineskip=15pt

\dottedcontents{section}[1.16cm]{}{1.8em}{5pt}
\dottedcontents{subsection}[2.00cm]{}{2.7em}{5pt}

\section{Introduction}
\setcounter{equation}{0}
\renewcommand{\theequation}
{1.\arabic{equation}}
\setcounter{theorem}{0}
\renewcommand{\thetheorem}
{1.\arabic{theorem}}

\noindent Yang-Baxter matrix equations  occupy a prominent place in pure mathematics and mathematical physics. Exploiting nontrivial solutions to a Yang-Baxter matrix equation over the complex field is a difficult task in general, whereas describing those solutions to some specific equations precisely is indispensable in applications to algebraic geometry and statistical mechanics. Compared to solving matrix equations over fields of characteristic zero, exploring solutions to a matrix equation over finite fields via formulating an explicit formula for the cardinality of all solutions has been more realizable computationally and indeed it has a long history with substantial ramifications in the study of combinatorics and algebra, dating back to, for example, \cite{Hod57,Hod58} and \cite{Hod64}. 
Our objectives of this article are to calculate the cardinality of solutions to a class of matrix equations over finite fields, and to study the geometry of the orbits of these solutions under the natural conjugation action of the general linear groups.

Let $\F$ be a field and $n\in\N^+$ be a positive integer. Given an $n\times n$ matrix $A$ over $\F$, the  quadratic matrix equation $A\cdot X\cdot A=X\cdot A\cdot X$
called the \textbf{parameter-independent Yang-Baxter equation} over $\F$, has been studied for the various cases where $\F$ is the field of complex numbers and $A$ possesses some special properties; see for example \cite{DD16,DDH18} and the references therein. Throughout this article, $\F=\F_q$ denotes the finite field of order $q=p^s$ and we are interested in solving the parameter-independent Yang-Baxter equation over $\F_q$, when $A=\diag\{a,\dots,a\}$ is a scalar diagonal matrix over $\F_q$.

To articulate some extreme situations, we let $\M(n,q)$ denote the vector space of all $n\times n$ matrices over $\F_q$.
If $A$ is the zero matrix (i.e., $a=0$), then each $X\in\M(n,q)$ is a solution. Now assume that $a\neq 0$. Since $A$ commutes with every matrix in $\M(n,q)$, we see that deciding whether $X\in\M(n,q)$ satisfies the parameter-independent Yang-Baxter equation is tantamount to verifying whether $X$ is a solution of the following equation:
\begin{equation}\tag{$\ast$}
\label{ast}
X^{2}-A\cdot X=0.
\end{equation}
We observe that the zero matrix and $A$ itself are both solutions of this equation; in particular, if $n=1$, the two solutions are all solutions as the left-hand side of (\ref{ast}) is a polynomial in one variable of degree 2 in this case. Moreover, we also observe that for any $n\in\N^+$, if $X$ is a nonsingular solution, then $X$ must be $A$. Denote by $\NN(n,q)$ the set of all solutions to (\ref{ast}) in $\M(n,q)$. 
Thus $|\NN(n,q)|-2$ is exactly equal to the number of nonzero singular solutions in $\M(n,q)$ and the difficulty in determining $|\NN(n,q)|$ is to find all nonzero singular  $n\times n$ matrices satisfying the equation (\ref{ast}).

The bulk of the first two sections is to calculate the cardinality of those nonzero singular solutions to (\ref{ast}). 
An elementary observation (Proposition \ref{prop2.1}) shows that $\NN(n,q)$ could be endowed with a conjugation action of the general linear group. This allows us to capitalize on the orbit-stabilizer formula and rational canonical forms of matrices 
to determine the number $|\NN(n,q)|$. After summarizing some preparations about classical conjugation invariants, 
rational canonical forms, and computational steps,  we close Section \ref{sec2} with an explicit calculation for the case where $n=2$; see Example \ref{exam2.2}. We will deal with the cases of higher dimensions ($n\geqslant 3$) in Section 3. 
To accomplish this, the key is to reveal the concrete form of the rational canonical form of a nonzero singular solution in $\NN(n,q)$; see Lemma \ref{lem3.2}. As a consequence (Corollary \ref{coro3.3}), we prove, via constructing representatives in orbits, that the cardinality of the set $\OO(n,q)$ of all orbits of $\NN(n,q)$ under the conjugation action is equal to $n+1$. 
Using the orbit-stabilizer formula, we finally derive an explicit formula on the cardinality $|\NN(n,q)|$; see Theorem \ref{thm3.6}.

In Section \ref{sec4}, we prove that the classical conjugation invariants $\xi_1,\dots,\xi_n$ separate the set 
$\OO(n,q)$ of orbits (Theorem \ref{thm4.3}).  Example \ref{exam4.4} hints at the potential universality of our approach of separating invariants in studying geometric properties of orbits. Consider the image points of these orbits in $\F_q^n$
under the injection defined by $\xi_1,\dots,\xi_n$. We find an ideal $\I_n$ of $\F_q[x_1,\dots,x_n]$, via giving explicit  generators, such that the variety of $\I_n$ in $\F_q^n$ coincides with the image of $\OO(n,q)$; see Theorem \ref{thm4.7}.
A surprising result appears in Proposition \ref{prop4.6}, showing that the ideal $\I_n$ could be generated by ${n+1\choose 2}$ quadratic polynomials. 

\subsection*{Conventions} Throughout this article, $\N^+$ denotes the set of all positive integers. Let $I_n$ be the identity matrix of rank $n\in\N^+$.
For $B\in\M(k,q)$ and $C\in\M(\ell,q)$, we use $B\oplus C$ to denote the block matrix 
$\left(\begin{smallmatrix}
      B&0 \\
     0&C  \\
\end{smallmatrix}\right)$ in $\M(k+\ell, q)$.

\subsection*{Acknowledgements} This research was partially supported by NNSF of China (No. 11401087). The
authors would like to thank the referee for a careful reading of the paper and for helpful suggestions.
The symbolic computation language MAGMA \cite{BCP97} (http://magma.maths.usyd.edu.au/) was very helpful.

\section{Conjugation Actions and Rational Canonical Forms} \label{sec2}
\setcounter{equation}{0}
\renewcommand{\theequation}
{2.\arabic{equation}}
\setcounter{theorem}{0}
\renewcommand{\thetheorem}
{2.\arabic{theorem}}

\noindent  In this preliminary section, we let $n\geqslant 2$ and $\GL(n,q)$ be the general linear group of degree $n$ over $\F_q$. Recall that the conjugation action of  $\GL(n,q)$ on $\M(n,q)$ is defined by $(P,X)\mapsto P\cdot X\cdot P^{-1}$
for $P\in\GL(n,q)$ and $X\in\M(n,q)$. We write $[X]$ for the conjugacy class of $X$.
Moreover, the characteristic polynomial of $X\in\M(n,q)$ is defined as
\begin{equation}
\label{ }
\det(\lambda\cdot I_n-X)=\lambda^n+\sum_{i=1}^n (-1)^i \cdot \xi_i(X)\cdot \lambda^{n-i}
\end{equation}
where $\lambda$ is an indeterminate and the coefficients $\xi_1,\xi_2,\dots,\xi_n$ are algebraically independent invariants in  the invariant ring $\F_q[\M(n,q)]^{\GL(n,q)}:=\{f\in\F_q[\M(n,q)]\mid P\cdot f=f, \textrm{for all }P\in \GL(n,q)\}$, where
$\F_q[\M(n,q)]$ denotes the coordinate ring of the $n^2$-dimensional affine space $\M(n,q)$ and 
$(P\cdot f)(X):=f(P^{-1}(X))=f(P^{-1}\cdot X\cdot P)$ for all $X\in \M(n,q)$. In particular, $\xi_1$ and $\xi_n$ are just the well-known trace and determinant functions respectively. 
Note that unlike the classical case (over the complex field), these $\xi_i$ here do not generate the invariant ring; see \cite[Theorem 1.1]{Smi02} for the case $n=2$.

The following  result indicates that the conjugation action of $\GL(n,q)$ on $\M(n,q)$  restricts to 
an action on $\NN(n,q)$. We denote by $\OO(n,q)$  the set of orbits of $\NN(n,q)$ under this action. 

\begin{prop}\label{prop2.1}
If an $n\times n$ matrix $X\in\NN(n,q)$, then $Y\in\NN(n,q)$ for all $Y\in[X]$.
\end{prop}

\begin{proof}
Suppose that $Y=P\cdot X\cdot P^{-1}$ for some $P\in \GL(n,q)$.
Since $X^{2}=A\cdot X$, we see  that $Y^{2}-A\cdot Y=(P\cdot X\cdot P^{-1})^{2}-A\cdot P\cdot X\cdot P^{-1}=P\cdot X^{2}\cdot P^{-1}-P\cdot A\cdot X\cdot P^{-1}=P\cdot (X^2-A\cdot X)\cdot P^{-1}=0$. Hence,
$Y\in\NN(n,q)$.
\end{proof}

Consider a monic polynomial $f(x)=x^{k}+ \sum_{i=0}^{k-1} a_i\cdot  x^i \in\F_{q}[x]$. The  companion matrix of $f(x)$ is defined as
\begin{equation}
\label{ }
C(f):=\begin{pmatrix}
      0&1&0&\cdots&0   \\
      0&0&1&\ddots&\vdots \\
      \vdots&\ddots&\ddots&\ddots&0  \\
       0&0&\cdots&0&1   \\
     -a_{0}&-a_{1}&-a_{2}&\cdots&-a_{k-1} \\
\end{pmatrix}
\end{equation}
for $k\geqslant 2$ and $C(f):=(-a_{0})$ for $k=1$. Recall that every matrix $X\in \M(n,q)$ is similar to a diagonal block matrix of the form $C(f_{1})\oplus C(f_{2})\oplus \dots \oplus C(f_{r})$, called the \textbf{rational canonical form} of $X$, 
where $f_{1}(x),\dots,f_{r}(x)\in\F_{q}[x]$ are monic polynomials and $f_{i}(x)$ divides $f_{i+1}(x)$ for $i=1,2,\dots,r-1$; see for example
\cite[Theorem 16.15]{Bro93}. By Proposition \ref{prop2.1}, to determine whether $X$ is in $\NN(n,q)$, we may assume that
$X=C(f_{1})\oplus C(f_{2})\oplus \dots \oplus C(f_{r})$
and further, we write $A=A_{1}\oplus A_{2}\oplus \dots\oplus A_{r}$ as a block matrix such that the sizes of $A_{i}$ and $C(f_{i})$ are same for each $i$.
Clearly, (\ref{ast}) is completely determined  by the system of equations:
\begin{equation}
\label{eq2.3}
C(f_{i})^{2}-A_{i}\cdot C(f_{i})=0
\end{equation}
for $i=1,2,\dots,r$.

Based on these observations, we may proceed the following steps to determine the cardinality 
$|\NN(n,q)|$, i.e., the number of solutions to (\ref{ast}). 
\begin{enumerate}
  \item Determine all possible nonzero singular rational canonical forms $X_{1},\dots,X_{t}$ of $n\times n$ matrices. 
  \item Find those $X_j$ from $\{X_{1},\dots,X_{t}\}$  for which the system (\ref{eq2.3}) of equations follows, and denote  by $X_{1},\dots,X_{\ell}$  (relabelling if necessary), where $\ell=|\OO(n,q)|-2$ and $\ell\leqslant t$.
  \item For $i\in\{1,\dots,\ell\}$, calculate the order of the stabilizer subgroup $\GL(n,q)_{X_{i}}$ of $X_{i}$ in $\GL(n,q)$. 
  Since the number of all nonzero singular solutions to (\ref{ast}) equals $\sum_{i=1}^{\ell} |[X_{i}]|$ and
  $|\GL(n,q)|=|[X_{i}]|\cdot |\GL(n,q)_{X_{i}}|$, it follows that
  \begin{equation}
\label{eq2.4}
|\NN(n,q)|=2+\sum_{i=1}^{\ell}  |[X_{i}]|=2+\sum_{i=1}^{\ell} \frac{|\GL(n,q)|}{|\GL(n,q)_{X_{i}}|}.
\end{equation}
\end{enumerate}

We conclude this section with the following example that not only illustrates the above procedure but also 
serves to higher dimension cases in Section \ref{sec3}. 

\begin{exam}[$n=2$]\label{exam2.2}
{\rm
There are two possible rational canonical forms:
$\left(\begin{smallmatrix}
      -a_{0}&  0  \\
      0& -a_{0}
\end{smallmatrix}\right)$ and $\left(\begin{smallmatrix}
     0 & 1   \\
     -a_{0} &-a_{1}
\end{smallmatrix}\right)$ for $a_0,a_1\in\F_q$. As the first canonical form is either zero or nonsingular, 
the second one is the unique canonical form for nonzero singular solutions. Note that its determinant is $a_0$, thus
$a_0=0$. This means that we may suppose $X_{1}=\left(\begin{smallmatrix}
     0 & 1   \\
     0 &-a_{1}
\end{smallmatrix}\right)$ is an arbitrary nonzero singular solution. 
Substituting $C(f_{i})$ in (\ref{eq2.3}) with $X_{1}$, we have
$$0=\begin{pmatrix}
     0 & 1   \\
     0 &-a_{1}
\end{pmatrix}^{2}-\begin{pmatrix}
     a & 0   \\
    0 &a
\end{pmatrix}\begin{pmatrix}
     0 & 1   \\
     0&-a_{1}
\end{pmatrix}=\begin{pmatrix}
    0  &  -a_{1}-a  \\
    0  &  a_{1}^{2}+aa_{1}
\end{pmatrix}$$
which implies that $X_{1}=\left(\begin{smallmatrix}
     0 & 1   \\
     0&a
\end{smallmatrix}\right)$. To determine $|[X_{1}]|$, we need to determine
the order of the stabilizer subgroup $\GL(2,q)_{X_{1}}$. Here we take a direct approach to do that. 
Let $P=\left(\begin{smallmatrix}
     e&b   \\
     d&c
\end{smallmatrix}\right)\in \GL(2,q)_{X_{1}}$ be any element. As $P\cdot X_{1}\cdot P^{-1}=X_{1}$, it follows that
$$0=\begin{pmatrix}
     e &b   \\
     d &c
\end{pmatrix}\begin{pmatrix}
     0 & 1   \\
     0&a
\end{pmatrix}-\begin{pmatrix}
     0 & 1   \\
     0&a
\end{pmatrix}\begin{pmatrix}
     e &b   \\
     d&c
\end{pmatrix}=\begin{pmatrix}
    -d&e+ab -c     \\
    -ad  &  d
\end{pmatrix}.$$
Thus $P=\left(\begin{smallmatrix}
     c-ab&b   \\
     0&c
\end{smallmatrix}\right)$. Since $P$ is invertible, we see that $c\neq 0$ and $b\neq c/a$.
Hence, $|\GL(2,q)_{X_{1}}|=(q-1)^{2}.$
Recall that $|\GL(2,q)|=(q^{2}-1)(q^{2}-q)$. Therefore
$$|[X_{1}]|=\frac{|\GL(2,q)|}{|\GL(2,q)_{X_{1}}|}=\frac{(q^{2}-1)(q^{2}-q)}{(q-1)^{2}}=q^{2}+q$$
and $|\NN(2,q)|=q^2+q+2.$ 
\hbo}\end{exam}

\section{$|\NN(n,q)|(n\geqslant 3)$} \label{sec3}
\setcounter{equation}{0}
\renewcommand{\theequation}
{3.\arabic{equation}}
\setcounter{theorem}{0}
\renewcommand{\thetheorem}
{3.\arabic{theorem}}

\noindent In this section, we will determine the cardinality of $\NN(n,q)$.
Let $n\geqslant 3$ and $X\in\M(n,q)$ be a matrix. Usually, it is difficult to determine the rational canonical form for $X$ precisely. However, with the assumption that $X\in\NN(n,q)$,  the following lemma shows that the canonical form of $X$ will be built by 
rational canonical blocks of size less than or equal to 2.

\begin{lem}\label{lem3.1}
Let $f(x)\in\F_{q}[x]$ be a monic polynomial of degree $k\geqslant 3$.
Then $C(f)^{2}\neq A\cdot C(f)$.
\end{lem}

\begin{proof}
A direct calculation shows that the entry at the first row and third column in $C(f)^{2}$ will be 1. However, the entry at the same position in $A\cdot C(f)$ is zero. Hence, $C(f)^{2}$ and $A\cdot C(f)$ are never equal.
\end{proof}

Note that we have determined nonzero singular rational canonical blocks of size 2 satisfying (\ref{eq2.3})  in Example \ref{exam2.2}.
Throughout this section, we let $Q(a):=\left(\begin{smallmatrix}
      0&1  \\
     0&a  \\
\end{smallmatrix}\right)$ and $P_m(b)$ be the diagonal matrix of size $m$ with the diagonals $b$ for $m\leqslant n$ and 
$b\in\F_q[\lambda]$, where $\lambda$ is an indeterminate. For $k\in\{1,\dots,\lfloor n/2\rfloor\}$, we define $Q_k(a)$ to be the direct sum of $k$ copies of $Q(a)$ and let
$$X_k(b,a):=P_{n-2k}(b)\oplus Q_k(a).$$
Note that when $n=2m$ is even, we make the convention that $X_{m}(b,a)=Q_{m}(a)$ for any $b$.

\begin{lem}\label{lem3.2}
Let $X\in\NN(n,q)$ be a nonzero singular matrix. Then $X$ is similar to either $X_k(0,a)$ or $X_k(a,a)$ for some $k\in\{1,\dots,\lfloor n/2\rfloor\}$.
\end{lem}

\begin{proof} We use $X_0$ to denote the rational canonical form of $X$. 
By Lemma \ref{lem3.1}, the blocks appeared in $X_0$ are of size either 1 or 2. 
If these blocks are all $1\times 1$, we may assume that $X_0=\diag\{-a_1,-a_2,\dots,-a_n\}$. Note that $x+a_i$ divides $x+a_{i+1}$ for $i=1,\dots,n-1$, thus $a_1=\dots=a_n$. Hence, $X_0$
is either zero or nonsingular, contradicting with that $X$ is nonzero singular. This means that $X_0$ contains at least one
block of size 2. Now we suppose that 
$$X_0=\diag\{-a_1,-a_2,\dots,-a_{n-2k}\}\bigoplus\left(\oplus_{i=1}^k \left(\begin{smallmatrix}
     0 & 1   \\
     -a_{i,0} &-a_{i,1}
\end{smallmatrix}\right)\right)$$
for some $k\in\{1,\dots,\lfloor n/2\rfloor\}$. As before, since $x+a_i$ divides $x+a_{i+1}$ for $i=1,\dots,n-2k-1$, it follows that $a_1=\dots=a_{n-2k}$. By (\ref{eq2.3}), we see that
$$X_0=P_{n-2k}(b)\bigoplus\left(\oplus_{i=1}^k \left(\begin{smallmatrix}
     0 & 1   \\
     -a_{i,0} &-a_{i,1}
\end{smallmatrix}\right)\right)$$
where $b=0$ or $a$. Let $f_i=x^2+a_{i,1}x+a_{i,0}$ be the polynomial with $C(f_i)=\left(\begin{smallmatrix}
     0 & 1   \\
     -a_{i,0} &-a_{i,1}
\end{smallmatrix}\right)$.  Since $f_{i+1}$ is divisible by $f_i$ for each $i=1,\dots,k-1$, we see that 
$f_1=f_2=\dots=f_k$. Thus
$$X_0=P_{n-2k}(b)\bigoplus\left(\oplus_{i=1}^k \left(\begin{smallmatrix}
     0 & 1   \\
     -a_{0} &-a_{1}
\end{smallmatrix}\right)\right)$$
for some $a_0,a_1\in\F_q$. Note that the polynomial corresponding the $(n-2k)$-th block of $X_0$ is $x$ and
the polynomial corresponding the $(n-2k+1)$-th block is $x^2+a_{1}x+a_{0}$. Being divisible by $x$ for $x^2+a_{1}x+a_{0}$ implies that $a_0=0$. Applying (\ref{eq2.3}) again, it follows from Example \ref{exam2.2} that 
$a_1=-a$.
Therefore, $X_0=X_k(b,a)$, where $b=0$ or $a$.
\end{proof}

\begin{coro}\label{coro3.3}
The cardinality of $\OO(n,q)$ is $n+1$.
\end{coro}

\begin{proof}
If $n=2m$ is even, then $\OO(n,q)=\{[P_{n}(0)],[P_{n}(a)],[Q_{m}(a)],[X_k(0,a)], [X_k(a,a)]\mid 1\leqslant k\leqslant m-1\}$.
Thus $|\OO(n,q)|=2(m-1)+3=n+1$. If $n=2m+1$ is odd, then 
$$\OO(n,q)=\{[P_{n}(0)],[P_{n}(a)],[X_k(0,a)], [X_k(a,a)]\mid 1\leqslant k\leqslant m\}.$$
Thus $|\OO(n,q)|=2m+2=n+1$.
\end{proof}

\begin{lem}\label{lem3.4}
If $k\in\{1,\dots,\lfloor n/2\rfloor\}$, then
\begin{enumerate}
  \item the elementary divisors of $X_k(0,a)$ consist of  $n-k$ copies of $\lambda$ and $k$ copies of $\lambda-a$; and
  \item the elementary divisors of $X_k(a,a)$ consist of $n-k$ copies of $\lambda-a$ and $k$ copies of $\lambda$.
\end{enumerate}
\end{lem}

\begin{proof} Here all $\lambda$-matrices involved will be working over the polynomial ring $\F_q[\lambda]$.
Note that as $a$ is invertible, the $\lambda$-matrix $\left(\begin{smallmatrix}
      \lambda&-1  \\
     0&\lambda-a  \\
\end{smallmatrix}\right)$ of $Q(a)$ could be diagonalized via applying elementary transformations. In fact, 
$\left(\begin{smallmatrix}
     1&-\frac{1}{a} \\
     0&1  \\
\end{smallmatrix}\right)\left(\begin{smallmatrix}
      \lambda&-1  \\
     0&\lambda-a  \\
\end{smallmatrix}\right)\left(\begin{smallmatrix}
     1&\frac{1}{a} \\
     0&1  \\
\end{smallmatrix}\right)=\left(\begin{smallmatrix}
      \lambda&0  \\
     0&\lambda-a  \\
\end{smallmatrix}\right)=:\widetilde{Q}(a)$.  Let $\widetilde{Q}_k(a)$ denote the direct sum of $k$ copies of 
$\widetilde{Q}(a)$.

(1) We first capitalize on \cite[Lemma 16.12]{Bro93} to find all invariant factors of $X_k(0,a)$. Clearly, the $\lambda$-matrix  $\lambda\cdot I_n-X_k(0,a)$ is equivalent to  $P_{n-2k}(\lambda)\oplus \widetilde{Q}_k(a)$, which has the Smith normal form diag$\{h_1(\lambda),h_2(\lambda),\dots,h_n(\lambda)\}$, we say. Here
$h_i(\lambda)$ divides $h_{i+1}(\lambda)$ for $i=1,2,\dots,n-1$.
Since the values of minors of order $i$ of $P_{n-2k}(\lambda)\oplus \widetilde{Q}_k(a)$ are of forms $\lambda^j(\lambda-a)^{i-j}$ with $0\leqslant j\leqslant i$, it follows that $h_1(\lambda)=\cdots=h_k(\lambda)=1$ and $h_{n}(\lambda)=\lambda^{n-k}(\lambda-a)^k$. For $k+1\leqslant i\leqslant n-1$, we have
$$h_{i}(\lambda)=\begin{cases}
    \lambda^{i-k},  & i\leqslant n-k, \\
    \lambda^{i-k}(\lambda-a)^{i-(n-k)}, & i>n-k.
\end{cases}$$
Hence, the invariant factors of 
$\lambda\cdot I_n-X_k(0,a)$ consist of 
$$\{\underbrace{1,\dots,1}_k,\underbrace{\lambda,\dots,\lambda}_{n-2k},\underbrace{\lambda(\lambda-a),\dots,\lambda(\lambda-a)}_k\}$$
and the elementary factors contains $n-k$ copies of $\lambda$ and $k$ copies of $\lambda-a$.

(2) Similarly, we note that   $\lambda\cdot I_n-X_k(a,a)$ is equivalent to  $P_{n-2k}(\lambda-a)\oplus \widetilde{Q}_k(a)$ and assume that the corresponding Smith normal form is diag$\{h_1(\lambda),h_2(\lambda),\dots,h_n(\lambda)\}$. Observe that $h_1(\lambda)=\cdots=h_k(\lambda)=1$ and $h_{n}(\lambda)=\lambda^{k}(\lambda-a)^{n-k}$.
Switching  the roles of $\lambda$ and $\lambda-a$ in the previous case, we see that  for $k+1\leqslant i\leqslant n-1$, 
$$h_{i}(\lambda)=\begin{cases}
    (\lambda-a)^{i-k},  & i\leqslant n-k, \\
    (\lambda-a)^{i-k}\lambda^{i-(n-k)}, & i>n-k.
\end{cases}$$
Hence, the corresponding elementary factors consist of $n-k$ copies of $\lambda-a$ and $k$ copies of $\lambda$. 
\end{proof}

\begin{coro}\label{coro3.5}
For each $k\in\{1,\dots,\lfloor n/2\rfloor\}$ and $b\in\{0,a\}$, we have
$$|\GL(n,q)_{X_k(b,a)}|=|\GL(n-k,q)|\cdot |\GL(k,q)|.$$
\end{coro}

\begin{proof}
It follows from Lemma \ref{lem3.4} and \cite[Theorem 6.14]{Hou18} that $$|\GL(n,q)_{X_k(b,a)}|=q^{(n-k)^2+k^2}\cdot\prod_{i=1}^{n-k}(1-q^{-i})\cdot\prod_{j=1}^{k}(1-q^{-j})$$
which is exactly equal to the product of the orders of $\GL(n-k,q)$ and $\GL(k,q)$.
\end{proof}

Together (\ref{eq2.4}), Corollary \ref{coro3.5} and Example \ref{exam2.2} immediately imply that

\begin{thm} \label{thm3.6}
For $n\geqslant 2$, we have
$$
|\NN(n,q)| =  2+ \begin{cases}
   |\GL(2m,q)|\cdot \left(\frac{1}{|\GL(m,q)|^2}+\sum_{k=1}^{m-1}\frac{2}{|\GL(2m-k,q)|\cdot |\GL(k,q)|}\right),   & n=2m, \\
  |\GL(2m+1,q)|\cdot \sum_{k=1}^{m}\frac{2}{|\GL(2m-k+1,q)|\cdot |\GL(k,q)|},   & n=2m+1,
\end{cases}
$$where $|\GL(\ell,q)|=\prod_{i=0}^{\ell-1}(q^\ell-q^i)$ for every $\ell\in\N^+$.
\end{thm}

Note that here the construction of orbits in Corollary \ref{coro3.3} has been applied. We conclude this section by showcasing  $|\NN(n,q)|$ for several  small $n$.

\begin{exam}
{\rm
\begin{enumerate}
  \item $|\NN(2,q)|=q^2+q+2$. 
  \item $|\NN(3,q)|=2q^2(q^2+q+1)+2$. 
  \item $|\NN(4,q)|=q^3(q^2+1)(q^3+q^2+3q+2)+2$. 
  \item $|\NN(5,q)|=2q^4(q^2-q+1)(q^2+q+1)(q^4+q^3+q^2+q+1)+2$. \hbo
\end{enumerate}
}\end{exam}
\section{Separating Invariants} \label{sec4}
\setcounter{equation}{0}
\renewcommand{\theequation}
{4.\arabic{equation}}
\setcounter{theorem}{0}
\renewcommand{\thetheorem}
{4.\arabic{theorem}}

\noindent In this section, we separate the orbits via invariants and find a generating set for the vanishing ideal of these orbits.
Consider the set $\OO(n,q)$ of orbits and the classical conjugation invariants $\xi_1,\xi_2,\dots,\xi_n$. 
The map $\xi:\OO(n,q)\ra\F_q^n$ given by  $X\mapsto (\xi_1(X),\xi_2(X),\dots,\xi_n(X))$ is well-defined.
Let $\A$ be the set of all functions from $\OO(n,q)$ to $\F_q$. We say that a subset $\B\subseteq \A$
is \textbf{separating} for $\OO(n,q)$ if for any two distinct orbits $X,Y\in\OO(n,q)$, there exists a function 
$f\in\B$ such that $f(X)\neq f(Y)$; see \cite[Section 2.4]{DK15}, \cite{CSW21, KLR20} and \cite{CCSW19} for more details and recent development on separating invariants.

\begin{lem}\label{lem4.1}
The map $\xi$ is injective if and only if $\{\xi_1,\xi_2,\dots,\xi_n\}$ is separating for $\OO(n,q)$.
\end{lem}

\begin{proof}
Assume that $\xi$ is injective and $\{\xi_1,\xi_2,\dots,\xi_n\}$ is not separating. Then there exist two distinct orbits $X,Y\in\OO(n,q)$ such that $\xi_i(X)=\xi_i(Y)$ for all $i=1,\dots,n$. Thus $\xi(X)=\xi(Y)$, which contradicts with the assumption that $\xi$ is injective. Conversely, if $\{\xi_1,\xi_2,\dots,\xi_n\}$ is separating, then for any two distinct orbits $X,Y\in\OO(n,q)$, there exists some $i\in\{1,\dots,n\}$ such that $\xi_i(X)\neq \xi_i(Y)$. Thus
$\xi(X)\neq\xi(Y)$ and $\xi$ is injective.
\end{proof}

\begin{lem}\label{lem4.2}
The cardinality of the image of $\xi$ is equal to $n+1$.
\end{lem}

\begin{proof}
Let $\varphi_\lambda(X):=\det(\lambda\cdot I_n-X)$ be the characteristic polynomial of a matrix $X\in\M(n,q)$.
Note that $\varphi_\lambda(P_n(0))=\lambda^n$ and $\varphi_\lambda(P_n(a))=(\lambda-a)^n$. Thus
$\xi([P_n(0)])=(0,0,\dots,0)$ and $\xi([P_n(a)])=\left(n\cdot a,{n\choose 2}\cdot a^2,\dots,a^n\right)$.
By Corollary \ref{coro3.3}, for each remaining orbit $[X]\in \OO(n,q)\setminus\{[P_n(0)],[P_n(a)]\}$, there exist some
$k\in \{1,\dots,\lfloor n/2\rfloor\}$ such that $\varphi_\lambda(X)=(\lambda-b)^{n-2k}\cdot\lambda^k\cdot (\lambda-a)^k$
where $b\in\{0,a\}$. 

Assume that $n=2m$ is even. For $1\leqslant k\leqslant m-1$, we have $\varphi_\lambda(X_k(0,a))=\lambda^{2m-k}\cdot (\lambda-a)^k$ and  $\xi([X_k(0,a)])=\left(k\cdot a,{k\choose 2}\cdot a^2,\dots,a^k,0,\dots,0\right)$.
Since $\varphi_\lambda(Q_m(a))=\lambda^m\cdot (\lambda-a)^m$, it follows that
$$\xi([Q_m(a)])=\left(m\cdot a,{m\choose 2}\cdot a^2,\dots,a^m,0,\dots,0\right).$$
Furthermore, as $\varphi_\lambda(X_k(a,a))=\lambda^{k}\cdot (\lambda-a)^{2m-k}$, we see that 
$$\xi([X_k(a,a)])=\left((2m-k)\cdot a,{2m-k\choose 2}\cdot a^2,\dots,a^{2m-k},0,\dots,0\right)$$
for $1\leqslant k\leqslant m-1$. Consider the ordered sequence 
$\xi([P_{2m}(0)]),\xi([X_1(0,a)]), \dots, \xi([X_{m-1}(0,a)])$, $\xi([Q_m(a)]),\xi([X_{m-1}(a,a)]),\dots,\xi([X_1(a,a)]), \xi([P_{2m}(a)])$. Arraying the last $2m$ items into rows, we obtain a $2m\times 2m$ lower triangular matrix:
$$\begin{pmatrix}
     a & 0& \cdots& 0  \\
      \ast&a^2& \ddots&\vdots \\
      \vdots&\ddots&\ddots&0\\
      \ast&\cdots&\cdots&a^{2m}
\end{pmatrix}$$
which is invertible as $a\neq 0.$ This fact shows that 
the map $\xi$ evaluating on $\OO(2m,q)$ has $2m+1$ distinct values.
A similar argument also applies to the case where $n=2m+1$ is odd. Finally, we conclude that 
the cardinality of the image of $\xi$ is equal to $n+1$.
\end{proof}

\begin{thm}\label{thm4.3}
The set $\{\xi_1,\xi_2,\dots,\xi_n\}$ is separating for $\OO(n,q)$. Moreover, if $p>n$, then
$\xi_1$ can separate orbits in $\OO(n,q)$. 
\end{thm}

\begin{proof}
Lemma \ref{lem4.2} together with Corollary \ref{coro3.3} implies that $\xi$ is injective. By Lemma \ref{lem4.1}, we see that 
$\xi_1,\xi_2,\dots,\xi_n$ separate the orbit set $\OO(n,q)$. For the second statement, we note in the proof of Lemma \ref{lem4.2} that $\{\xi_1(X)\mid X\in\OO(n,q)\}=\{k\cdot a\mid 0\leqslant k\leqslant n\}$, which has cardinality $n+1$, by 
the assumption $p>n.$ Hence, $\xi_1$ separates all orbits in $\OO(n,q)$.
\end{proof}

The following example illustrates that when $p\leqslant n$, $\xi_2,\dots,\xi_n$ might not be superfluous.  

\begin{exam}\label{exam4.4}
{\rm
Suppose that $n=3$. Then $\OO(3,q)=\{[P_3(0)], [X_1(0,a)], [X_1(a,a)], [P_3(a)]\}$ and the values of
$\xi$ on elements of $\OO(3,q)$ are: $(0,0,0), (a,0,0), (2a,a^2,0), (3a,3a^2,a^3)$ respectively. 

(1) If $p=2$, then either $\xi_1$ or $\xi_3$ can not separate the orbits $[P_3(0)]$ and $[X_1(a,a)]$. 
Hence, $\xi_2$ can not be removed in this case. However, $\xi_3$ is superfluous. In fact, the map
$$\OO(3,q)\ra \F_q^2, \quad X\mapsto (\xi_1(X),\xi_2(X))$$
is injective. Via this injection, we observe that orbits of $\OO(3,q)$ forms a rectangle in the plane $\F_q^2$:
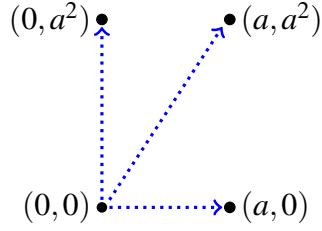
\begin{center}
\begin{tikzpicture}
\fill (-1,-1) circle (.07);
\node at (-1,-1) [left] {$(0,0)$};

\draw[fill] (-1,1.5) circle (.07);
\node at (-1,1.5) [left] {$(0,a^2)$};

\draw[fill] (0.7,-1) circle (.07);
\node at (0.7,-1) [right] {$(a,0)$};

\draw[fill] (0.7,1.5) circle (.07);
\node at (0.7,1.5) [right] {$(a,a^2)$};

\draw[color=blue, very thick, dotted, ->]  (-0.9,-1) -- (0.6,-1);
\draw[color=blue,  very thick, dotted, ->]  (-1,-0.9) -- (-1,1.4);
\draw[color=blue,  very thick, dotted, ->]  (-0.9,-0.9) -- (0.6,1.4);

\end{tikzpicture}
\end{center}
where the four points $(0,0), (a,0), (0,a^2), (a,a^2)$ correspond to $[P_3(0)], [X_1(0,a)], [X_1(a,a)], [P_3(a)]$ in $\OO(3,q)$
respectively.

(2) Assume that $p=3$. The functions $\xi_1$ and $\xi_2$ can not separate the orbits $[P_3(0)]$ and $[P_3(a)]$. Thus $\xi_3$ is necessary in this case. After removing $\xi_2$ in $\xi$, the injective map $\OO(3,q)\ra \F_q^2$ defined by $X\mapsto (\xi_1(X),\xi_3(X))$ embeds $\OO(3,q)$ into an isosceles triangle in $\F_q^2$:
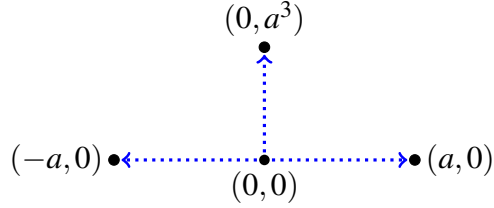
\begin{center}
\begin{tikzpicture}
\fill (0,0) circle (.07);
\node at (0,0) [below] {$(0,0)$};

\draw[fill] (0,1.5) circle (.07);
\node at (0,1.5) [above] {$(0,a^3)$};

\draw[fill] (2,0) circle (.07);
\node at (2,0) [right] {$(a,0)$};

\draw[fill] (-2,0) circle (.07);
\node at (-2,0) [left] {$(-a,0)$};

\draw[color=blue, very thick, dotted, ->]  (0,0.1) -- (0,1.4);
\draw[color=blue,  very thick, dotted, ->]  (0.1,0) -- (1.9,0);
\draw[color=blue,  very thick, dotted, ->]  (-0.1,0) -- (-1.9,0);
\end{tikzpicture}
\end{center}
where  $(0,0), (a,0), (-a,0), (0,a^3)$ correspond to $[P_3(0)], [X_1(0,a)], [X_1(a,a)], [P_3(a)]$ in $\OO(3,q)$
respectively. 
\hbo
}\end{exam}

We look back at the image points of $\OO(n,q)$ in $\F_q^n$ via the map $\xi$. As in the proof of Lemma \ref{lem4.2}, we use $v_0,v_1,\dots,v_n$ to denote these points respectively. More precisely, $v_0=(0,0,\dots,0)$ and
$$v_1=(a,0,\dots,0),v_2=\left({2\choose 1}a,{2\choose 2}a^2,0,\dots,0\right),\dots, v_n=\left({n\choose 1}a,{n\choose 2}a^2,\dots,{n\choose n}a^n\right).$$
The rest of this section is devoted to finding an  ideal $\I_n\subseteq \A_n:=\F_q[x_1,\dots,x_n]$ such that $\V_n:=\{v_i\mid i=0,1,\dots,n\}$ is the variety (i.e., set of zeros) of $\I_n$ in $\F_q^n.$ Throughout we denote by $V(\I_n)$ the variety of
$\I_n$.

We start with the case $n=2$.

\begin{prop}\label{prop4.5}
Let $\I_2$ be the ideal of $\A_2$ generated by 
$$\B_2:=\{f_{22}:=x_2^2-a^2x_2, f_{21}:=x_2x_1-2ax_2, f_{11}:=x_1^2-ax_1-2x_2\}.$$
Then $V(\I_2)=\V_2$.
\end{prop}

\begin{proof}
Assume that $v=(c_1,c_2)\in V(\I_2)$ is any element. Since $f_{22}(v)=0$, we see that $c_2$ is equal to either 0 or $a^2$. If $c_2=0$, then $f_{21}(v)=0$ for any $c_1$, and the fact that $f_{11}(v)=0$ implies that 
$c_1\in\{0,a\}$. If $c_2=a^2$, it follows from the fact that $f_{21}(v)=0$ that $c_1=2a$. Clearly, the valuation of $f_{11}$ at
$(2a,a^2)$ is zero. Hence, $V(\I_2)=\{(0,0), (a,0),(2a,a^2)\}=\V_2$.
\end{proof}

We regard $\A_1\subseteq\cdots\subseteq\A_k\subseteq \A_{k+1}\subseteq\cdots\subseteq\A_n$ as a sequence of containments of $\F_q$-subalgebras of $\A_n$.
For $n\geqslant 3$, we define 
$$\B_n:=\left\{f-\frac{f(w_n)}{a^n}\cdot x_n\mid f\in\B_{n-1}\right\}\cup\left\{x_n\cdot x_{i}-{n\choose i}a^{i}\cdot x_n\mid i=1,\dots,n\right\}$$
where $w_{n}:=\left({n\choose 1}a,{n\choose 2}a^2,\dots,{n\choose n-1}a^{n-1}\right)\in\F_q^{n-1}$. 

\begin{prop}\label{prop4.6}
For each $n\geqslant 2$, we have  $|\B_n|={n+1\choose 2}$.
\end{prop}

\begin{proof} 
We may assume that $n\geqslant 3$ as the case $n=2$ follows from Proposition \ref{prop4.5}.
Note that every $f\in\B_{n-1}$ does not involve  $x_n$. By the definition of $\B_n$, we see that $|\B_n|=|\B_{n-1}|+n$.
Since the induction hypothesis implies that $|\B_{n-1}|={n\choose 2}$, it follows that 
$|\B_n|={n\choose 2}+{n\choose 1}={n+1\choose 2}$.
\end{proof}

For example,
when $n=3$, we see that $w_3=(3a,3a^2)$, $|\B_3|=6$ and  
$$\B_3=\{f_{11},f_{21}-3x_3,f_{22}-6a\cdot x_3\}\cup\{x_3^2-a^3\cdot x_3, x_3x_2-3a^2\cdot x_3,x_3x_1-3a\cdot x_3\}.$$

\begin{thm}\label{thm4.7}
Let $n\geqslant 2$ and $\I_n$ be the ideal of $\A_n$ generated by $\B_n$. Then $V(\I_n)=\V_n$.
\end{thm}

\begin{proof}
We may assume that $n\geqslant 3$. Given a vector $v\in\F_q^n$, we denote by $\wt{v}$ the projection image of $v$ onto
$\F_q^{n-1}$ via removing the last component of $v$. We first show that each $v_i\in\V_n$ belongs to
$V(\I_n)$. Indeed, for $1\leqslant i\leqslant n$, we see that the valuation 
$(x_n\cdot x_{i}-{n\choose i}a^{i}\cdot x_n)\mid_{v_n}=a^n\cdot {n\choose i}a^{i}-{n\choose i}a^{i}\cdot a^n=0.$
Further, for $f\in\B_{n-1}$, note that $f$ does not involve $x_n$ and $\wt{v_n}=w_n$, thus
$(f-\frac{f(w_n)}{a^n}\cdot x_n)\mid_{v_n}=f(v_n)-f(w_n)=f(\wt{v_n})-f(w_n)=0.$
This shows that $v_n\in V(\I_n)$. Moreover, since the last components of $v_0,v_1,\dots,v_{n-1}$ are are zero and 
 the induction hypothesis implies that $\{\wt{v_i}\mid i=0,1,\dots,n-1\}\subseteq V(\I_{n-1})$, 
we deduce that the valuation of each $f\in\B_n$ at $v_i$ is equal to zero for  $i\in\{0,1,\dots,n-1\}$.
This proves that $\V_n\subseteq V(\I_n)$.

Conversely, since $|\V_n|=n+1$, it suffices to show that $|V(\I_n)|=n+1$.
Suppose  $v=(c_1,c_2,\dots,c_n)$ $\in V(\I_n)$ denotes an arbitrary element. 
Since $x_n^2-a^n\cdot x_n\in\B_n$, it follows that $c_n^2-a^n\cdot c_n=0$, which implies that $c_n$ must be in $\{0,a^n\}$. If $c_n=a^n$, then $v=v_n\in V(\I_n)$ is unique; and assume that $c_n=0$, then   
$v\in V(\I_n)$ if and only if $\wt{v}\in V(\I_{n-1})$. Thus $|V(\I_n)|=|V(\I_{n-1})|+1=n+1$, as desired. Here 
the last equation holds from the induction hypothesis that $|V(\I_{n-1})|=n$.
\end{proof}


\raggedright

\begin{bibdiv}
  \begin{biblist}
  
  \bib{BCP97}{article}{
   author={Bosma, Wieb},
   author={Cannon, John},
   author={Playoust, Catherine},
   title={The Magma algebra system. I. The user language},
   journal={J. Symbolic Comput.},
   volume={24},
   date={1997},
   number={3-4},
   pages={235--265},
   issn={0747-7171},
}
  
 \bib{Bro93}{book}{
   author={Brown, William C.},
   title={Matrices over commutative rings},
   series={Monographs and Textbooks in Pure and Applied Mathematics},
   volume={169},
   publisher={Marcel Dekker, Inc., New York},
   date={1993},
   pages={viii+281},
   isbn={0-8247-8755-2},
}

\bib{CCSW19}{article}{
   author={Campbell, H. E. A.},
   author={Chuai, J.},
   author={Shank, R. James},
   author={Wehlau, David L.},
   title={Representations of elementary abelian $p$-groups and finite
   subgroups of fields},
   journal={J. Pure Appl. Algebra},
   volume={223},
   date={2019},
   number={5},
   pages={2015--2035},
   issn={0022-4049},
}

\bib{CSW21}{article}{
   author={Chen, Yin},
   author={Shank, R. James},
   author={Wehlau, David L.},
   title={Modular invariants of finite gluing groups},
   journal={J. Algebra},
   volume={566},
   date={2021},
   pages={405--434},
   issn={0021-8693},
}

\bib{DK15}{book}{
   author={Derksen, Harm},
   author={Kemper, Gregor},
   title={Computational invariant theory},
   series={Encyclopaedia of Mathematical Sciences},
   volume={130},
   edition={Second enlarged edition},
   publisher={Springer, Heidelberg},
   date={2015},
   pages={xxii+366},
}



\bib{DD16}{article}{
   author={Dong, Qixiang},
   author={Ding, Jiu},
   title={Complete commuting solutions of the Yang-Baxter-like matrix
   equation for diagonalizable matrices},
   journal={Comput. Math. Appl.},
   volume={72},
   date={2016},
   number={1},
   pages={194--201},
   issn={0898-1221},
}

\bib{DDH18}{article}{
   author={Dong, Qixiang},
   author={Ding, Jiu},
   author={Huang, Qianglian},
   title={Commuting solutions of a quadratic matrix equation for nilpotent
   matrices},
   journal={Algebra Colloq.},
   volume={25},
   date={2018},
   number={1},
   pages={31--44},
   issn={1005-3867},
}








\bib{Hod57}{article}{
   author={Hodges, John H.},
   title={Some matrix equations over a finite field},
   journal={Ann. Mat. Pura Appl. (4)},
   volume={44},
   date={1957},
   pages={245--250},
   issn={0003-4622},
}

\bib{Hod58}{article}{
   author={Hodges, John H.},
   title={The matrix equation $X^{2}-I=0$ over a finite field},
   journal={Amer. Math. Monthly},
   volume={65},
   date={1958},
   pages={518--520},
   issn={0002-9890},
}

\bib{Hod64}{article}{
   author={Hodges, John H.},
   title={A bilinear matrix equation over a finite field},
   journal={Duke Math. J.},
   volume={31},
   date={1964},
   pages={661--666},
   issn={0012-7094},
}

\bib{Hou18}{book}{
   author={Hou, Xiang-dong},
   title={Lectures on finite fields},
   series={Graduate Studies in Mathematics},
   volume={190},
   publisher={American Mathematical Society, Providence, RI},
   date={2018},
   pages={x+229},
   isbn={978-1-4704-4289-7},
}

\bib{KLR20}{article}{
   author={Kemper, Gregor},
   author={Lopatin, Artem},
   author={Reimers, Fabian},
   title={Separating invariants over finite fields},
   journal={arXiv: 2011. 07408},
   date={2020},
}

  
\bib{Smi02}{article}{
   author={Smith, Larry},
   title={Invariants of $2\times 2$-matrices over finite fields},
   journal={Finite Fields Appl.},
   volume={8},
   date={2002},
   number={4},
   pages={504--510},
   issn={1071-5797},
}


  \end{biblist}
\end{bibdiv}
\end{document}